\documentclass[11pt,a4paper]{amsart}
\usepackage{amsmath,amssymb,amscd,amsthm,amsxtra,array}
\usepackage[usenames,dvipsnames]{xcolor}
\usepackage[
colorlinks=true,
linkcolor=blue,
citecolor=blue,
urlcolor=blue,
]{hyperref}

\usepackage{soul}

\newtheorem{theorem}{Theorem}[section]
\newtheorem{lemma}[theorem]{Lemma}

\newtheorem{corollary}[theorem]{Corollary}

\theoremstyle{definition}

\newtheorem{example}[theorem]{Example}

\theoremstyle{remark}
\newtheorem{remark}[theorem]{Remark}
\numberwithin{equation}{section}
\DeclareMathOperator*{\esssup}{ess\,sup}

\def\Xint#1{\mathchoice
  {\XXint\displaystyle\textstyle{#1}}%
  {\XXint\textstyle\scriptstyle{#1}}%
  {\XXint\scriptstyle\scriptscriptstyle{#1}}%
  {\XXint\scriptscriptstyle\scriptscriptstyle{#1}}%
  \!\int}
\def\XXint#1#2#3{{\setbox0=\hbox{$#1{#2#3}{\int}$}
    \vcenter{\hbox{$#2#3$}}\kern-.5\wd0}}

\def\avgint{\Xint-}

\numberwithin{equation}{section}

\begin{document}
\title[Reverse H\"older for general $A^*_p(\mu)$]{Reverse H\"older property for strong weights and general measures}

\author{Teresa Luque}
\address{Instituto de Ciencias Matem\'aticas CSIC-UAM-UC3M-UCM, C/ Nicol\'as Ca\-bre\-ra, 13-15, 28049 Madrid, Spain} 
\email{teresa.luque@icmat.es}

\author{Carlos P\'erez}
\address{Department of Mathematics, University of the Basque Country, Ikerbasque and BCAM, 48080, Spain} 
\email{carlos.perezmo@ehu.es }

\author{Ezequiel Rela}
\address{Departamento de Matem\'atica,
Facultad de Ciencias Exactas y Naturales,
Universidad de Buenos Aires, Ciudad Universitaria
Pabell\'on I, Buenos Aires 1428 Capital Federal Argentina} \email{erela@dm.uba.ar}

\thanks{ The first author is is supported by  the Severo Ochoa Excellence Programme and the second author is supported by the Spanish Government grant MTM2014-53850-P and the Severo Ochoa Excellence Programme. The third author is partially supported by grants UBACyT 20020130100403BA, and PIP (CONICET) 11220110101018}



\begin{abstract}
We present dimension-free reverse H\"older inequalities for strong $A^*_p$ weights, $1\le p < \infty$. We also provide a proof for the full range of local integrability of $A_1^*$ weights. The common ingredient is a multidimensional version of Riesz's ``rising sun'' lemma.  Our results are valid for any nonnegative Radon measure with no atoms. For $p=\infty$, we also provide a reverse H\"older inequality for certain product measures. As a corollary we derive mixed $A_p^*-A_\infty^*$ weighted estimates.

\end{abstract}

\maketitle

\section{Introduction and Main Results}
In this article we present several results regarding reverse H\"older inequalities for strong $A^*_p(\mu)$ Muckenhoupt weights on $\mathbb{R}^n$ for a general non-atomic Radon measure $\mu$. Before describing these facts, a few words concerning the family of weights considered in here are necessary.

The class $A^*_p(\mu)$ of strong weights consist of all nonnegative $\mu$-measurable functions on $\mathbb{R}^n$ such that, for $1<p<\infty$ and $p'=p/(p-1)$, satisfy
\begin{equation}\label{eq:A^*p}
[w]_{A^*_p(\mu)}:=\sup_R \left(\avgint_R w\,d\mu\right) \left(\avgint_R w^{1-p'}\,d\mu\right)^{p-1} <\infty,
\end{equation}
where the supremum is taken over all rectangles $R\subset \mathbb{R}^n$ with sides parallel to the coordinate axes. As usual, we denote by $\avgint_E f \ d\mu=f_E=\frac{1}{\mu(E)}\int_E f \ d\mu$ the average of $f$ over $E$ with respect to the measure $\mu$. 

The limiting case of \eqref{eq:A^*p}, when $p=1$, defines the class $A_1^*(\mu)$; that is, the set of weights $w$ such that
\begin{equation*}%
[w]_{ A^*_1(\mu)  }:=\sup_{R}\bigg(\avgint_R w\,d\mu \bigg) \esssup_{R} (w^{-1})<+\infty.
\end{equation*}

This is equivalent to $w$ having the property
\begin{equation*}
 M_sw(x)\le [w]_{A_1^*(\mu)}w(x)\qquad \mu\text{ a.e. } x \in \mathbb{R}^n.
\end{equation*}
Here $M_s$ denotes the strong maximal function:
\begin{equation}\label{eq:M-strong-mu}
M_sf(x) = \sup_{R\ni x } \avgint_{R} |f|\ d\mu, 
\end{equation}
where the supremum is taken over all rectangles $R\subset \mathbb{R}^n$ with sides parallel to the coordinate axes containing the point $x$. Similarly, $M$ will denote the Hardy-Littlewood maximal function, namely when the supremum is taken over cubes with sides paralell to the coordinate axes.

It follows from H\"older's inequality and the definitions above that the classes $A^*_p(\mu)$ are increasing in $p\geq 1$. It is thus natural to define the limiting class $A^*_{\infty}(\mu)$ as
\begin{equation}\label{eq:A^*inf}
A^*_{\infty}(\mu):=\bigcup_{p\geq 1} 	A^*_p(\mu).
\end{equation}

If we only consider cubes with sides parallel to the coordinate axes, we obtain the classical Muckenhoupt $A_p(\mu)$ classes. Throughout the paper, we will often use the shorthand notations $A_p$ and $A_p^{*}$ since the underlying measure will always be clear from the context.

Typically, for $w\in A_p$, $1\le p \le \infty$, one expects an inequality of the form 
\begin{equation}\label{eq:general-RHI-intro}
\avgint_{Q} w^{1+\varepsilon}\ d\mu \le C\left(\avgint_{Q} w \ d\mu\right)^{1+\varepsilon},
\end{equation}
valid for any cube $Q$, where the constant $C$ may depend on the $A_p$ constant of the weight, the value of $\varepsilon$ and on the measure $\mu$. Inequalities like \eqref{eq:general-RHI-intro} are known as reverse H\"older inequalities (RHI)
and its study can be traced back to the works of Muckenhoupt \cite{Muckenhoupt:Ap}, Coifman and Fefferman \cite{CF}, within the context of harmonic analysis, and in the work of Gehring \cite{Gehring}, within the context of the theory of quasiconformal mappings. Since then, these kind of inequalities have been widely studied in many different situations with several motivations; we refer to \cite{AIM}  and \cite{IM} for the applications to elliptic PDE and quasiconformal mappings in the plane. We refer the interested reader to the monographs \cite[Chapter 4]{GCRdF} and \cite[Chapter 6 ]{ACS} for a more detailed information on these issues. More recently, RHI with good control on the constants have become relevant in the study of sharp bounds for some of the main operators in harmonic analysis such as singular integrals, maximal functions, commutators with BMO functions and others. See, for instance, \cite{LOP1} \cite{LOP2}  \cite{HP} \cite{CPP} for an account of this subject. More pre
 cisely, within this last context, it is particularly interesting to provide a version of such inequalities with a constant $C$ independent of the weight. 
In particular, in \cite{HPR1} the authors proved that, with underlying  Lebesgue measure on $\mathbb{R}^n$, we can take $C=2$ in the above inequality and the result is valid for any $w\in A_{\infty}$ and $0<\varepsilon \le\frac{1}{2^{n+1}[w]_{A_{\infty}}}$, where 
\begin{equation*}
 [w]_{A_\infty}:=\sup_Q\frac{1}{w(Q)}\int_Q M(w\chi_Q )\ dx <\infty
\end{equation*}
is called the Fujii-Wilson constant. Moreover, this result can be trivially extended to any doubling measures\ $\mu$ on $\mathbb{R}^n$; that is a measure $\mu$ such that:
\begin{equation*}\label{eq:doubling}
 \mu(2Q) \le C_\mu\,\mu(Q)
\end{equation*}
for every cube $Q$. Although the concept of doubling is affected by the family of sets we consider, in the particular case of rectangles the above definition remains the same. See \cite[Section 5]{HagLuPar} for a more complete definition of doubling measures on general basis.

Most of the known proofs of RHI for classical $A_p$ weights are based on the Calder\'on-Zygmund (C-Z) decomposition lemma applied to the level set of the Hardy-Littlewood maximal function $M$. 
This kind of stopping time argument produces a family of maximal cubes with nice properties. But in order to exploit the maximality, one needs to relate the average over some cube $Q$ to the average over the dyadic parent of $Q$.  It the case of doubling measures, the dilation process produces a dependence on the doubling constant. In particular, it forces the dependence on the dimension for the classical situation of the Lebesgue measure. This dependence on the dimension appears on the range of possible values for the exponent in the RHI.

In the context of non-doubling measures, \cite[Lemma 2.3]{OP-nondoubling} presents a characterization of the class $A_\infty=\bigcup_{p\geq 1} A_p$ in terms of several equivalent properties, and the RHI is among them. The only requirement imposed to the measure $\mu$ is the ``polynomial growth'' condition; that is, there exists some $0<\alpha\le n$ such that, for any $x\in \mathbb{R}^n$, and for any $r>0$,
\begin{equation}\label{eq:poli-growth}
\mu(B(x,r))\le C r^\alpha.
\end{equation} 
The relevant consequence of this condition is that $\mu$ does not concentrate positive measure on hyperplanes parallel to the coordinate axes of some system of coordinates. By changing variables, we can assume that this conditions is fulfilled for the canonical coordinates). Moreover, this latter condition is in fact a consequence of the absence of atoms (see \cite{MMNO}). In that case, the lack of doubling is solved in \cite[Lemma 2.1]{OP-nondoubling} by using a suitable version of the Besicovitch's covering theorem. This result provides a family of quasi-disjoint cubes with controlled average that covers the level set of the maximal function. The overlap is controlled by a dimensional constant $B(n)$, known as the Besicovitch constant. Tracking the constants in \cite[Lemma 2.3]{OP-nondoubling}, the following inequality holds
\[
\avgint_Q w^{1+\varepsilon}\ d\mu\leq2\left(\avgint_Q w\ d\mu\right)^{1+\varepsilon}
\]
for any $0<\varepsilon\leq\frac{1}{2^{p+1}B(n)[w]_{A_p}}$, $p>1$. Then, in this case, we observe that the RHI depends on the dimension via the Besicovitch constant.

In this paper, we focus our attention in the reverse H\"older property for strong weights. In the case of the Lebesgue measure, a simple change of variables produces a RHI for rectangles since it is known for cubes (the details can be found in \cite{DMRO-Ainfty}). However, this argument produces the same range for the exponent, and therefore it will appear a dependence on the dimension. This clashes with the somehow intuitive idea that, in many circumstances, strong weights behave like one dimensional objects; see \cite{Kurtz_multiparameter}  for further details on this issue. In addition, our purpose here is to investigate \eqref{eq:general-RHI-intro} for arbitrary non-atomic measures for which it is (in general) not possible to apply a change of variables argument.

We present a proof of a dimension-free RHI avoiding the use of C--Z type lemmas. We use instead what is known as a multidimensional form of the classical F. Riesz's ``Rising Sun'' lemma. The following lemma is from \cite{KLS}, and can be understood as a more precise version of the classic C--Z lemma.

\begin{lemma}[Multidimensional F. Riesz's lemma] \label{lem:MultiRiesz}
Let $R\subset \mathbb{R}^n$ be a rectangle and let $\mu$ be any non-atomic, nonnegative Radon measure on $R$. Let $f\in L^1_R(\mu)$ and $f_R\le \lambda$. Then there is a finite or countable set of pairwise disjoint rectangles $\{R_j\}_j$ for which 
\begin{equation*}
\avgint_{R_j}f\ d\mu =\lambda,
\end{equation*}
and $f(x)\le \lambda$ for $\mu$-almost all points $x\in R\setminus \left(\cup_j R_j\right)$.
\end{lemma}
We remark that in \cite{KLS} the above lemma is formulated for absolutely continuous measures. But an inspection of the proof shows that it is sufficient that the measure satisfies $\mu(L)=0$  for any hyperplane parallel to the coordinate axes. 

Our first main result is the following.
\begin{theorem}\label{thm:RHI-dim-free} Consider a non-negative, non-atomic Radon measure $\mu$. Let $w\in A^*_p$, $1< p<\infty$, and let $R$ be a rectangle.  Then
\[
\avgint_R w^{1+\varepsilon}\ d\mu \leq2\left(\avgint_R w\ d\mu\right)^{1+\varepsilon}.
\]
for any $0<\varepsilon\leq\frac{1}{2^{p+2}[w]_{A^*_p}}.$ 

\end{theorem}

For the particular case of the $A_1^*$ weights, we also address other questions regarding the sharp local integrability range for the weight $w$. In dimension 1 (for the Lebesgue measure), it is known that if a weight $w$ is in $A_1^*\equiv A_1$ then for any finite interval $I\subset \mathbb{R}$ we have that
\begin{equation}\label{eq:fullA1-dim1}
\avgint_I w(x)^s\ dx \le C_{s,w}\left(\avgint_I w(x) \ dx  \right)^s 
\end{equation}
for all $s$ such that $1<s<\frac{[w]_{A_1}}{[w]_{A_1}-1}=([w]_{A_1})'$. This result is from \cite[Corollary 1]{BSW} and there is also a sharp estimate on $C_{s,w}$ (see also \cite{malak2001} and \cite{malak2002}). In higher dimensions the known result is due to Kinnunen.  In \cite{Kinnunen-Dissertationes} the author proved the analogue of \eqref{eq:fullA1-dim1} for $A^*_1$ weights, also with sharp constants. His proof  relies strongly on the fact that the Lebesgue measure is a product measure. Therefore, an induction argument on the dimension can be carried out. For the particular case of cubes, the best known result is in \cite[Theorem 1.3]{Kinnunen-PubMat}. 

In the case of dyadic $A_1$ weights, the sharp result can be found in \cite{Melas}. For the case of doubling measures in metric spaces, some results are in \cite{AB} but without sharp constants.

Our second main result is the extension of Kinnunen's result to general measures:

\begin{theorem}\label{thm:full-range-A1}
Let $\mu$ be a non-atomic Radon measure on $\mathbb{R}^n$. Let $w\in A^*_1$. Then
\[
\avgint_R w^s \ d\mu \leq \frac{s}{1-(s-1)([w]_{A^*_1}-1)} \left(\avgint_R w\ d\mu\right)^s,
\]
for any $1<s<\frac{[w]_{A^*_1}}{[w]_{A^*_1}-1}.$
\end{theorem}

We present here a short and simple proof based on the multidimensional Riesz's lemma to obtain the result for  general measures with mild conditions. With this argument we are able to obtain the same optimal range for the exponent $s$.
Moreover,  Theorem \ref{thm:RHAp-maximal} describes a different version of these results with the range $s$ depending on the norm of 
the strong maximal operator $M_s$  on the dual space, namely $L^{p'}(w^{1-p'})$. In fact the case $p=1$ can be seen as a limiting case of  Theorem \ref{thm:RHAp-maximal}.

Finally, we also study reverse H\"older inequalities for strong $A^*_\infty$ weights where the  previous approach using the Riesz's Lemma cannot be extended for general measures. For the particular case of the Lebesgue measure, \cite{HagPar-SolyanikMaximal} presents also a different nice approach using Solyanik estimates.

This article is organized as follows. In Section \ref{sec:dim-free-strong} we prove Theorem \ref{thm:RHI-dim-free}. In Section \ref{sec:full-range} we show that a similar argument can be used to derive Theorem \ref{thm:full-range-A1} and obtain the full range of local integrability for $A_1^*$ weights. In Section \ref{sec:Ainfty} we study this problem for $A_\infty^*$ weights. 
Finally, in Section \ref{sec:further} we study different formulations of RHI for $A^*_p$ weights.

\section{Dimension-free RHI for \texorpdfstring{$A^*_p$}{Ap}, \texorpdfstring{$1\le p<\infty $}{1<p<infty} }\label{sec:dim-free-strong}

In this section we prove Theorem \ref{thm:RHI-dim-free}. We start with the following lemma, valid for $A^*_p$ weights for $p\in (1,\infty)$. 

\begin{lemma}\label{lem:weak11}
Let $\mu$ be a non-atomic Radon measure $\mu$. Let $w\in A^*_p$, $1<p<\infty$. 
Then, for any rectangle $R$ and any $\lambda>w_R$, we   have that 
\begin{equation}\label{eq:prop_2}
w(\{x\in R:w(x)>\lambda \})\leq 2\lambda \mu(\{x\in R: w(x)>\frac{1}{2^{p-1}[w]_{A^*_p}}w_R\}).
\end{equation}
\end{lemma}
\begin{proof}
Using H\"older's inequality with $p$ and its conjugate $p'$, we have that for every rectangle $R$ and every $f\geq 0$,
\[
\left(\avgint_R f\ d\mu\right)^pw(R)\leq[w]_{A^*_p}\int_R f^pw\ d\mu.
\]
In particular, for any $\mu$-measurable set $E\subset R$ we can rewrite the last inequality  for $f\equiv\chi_{E}$
\begin{equation}\label{eq:prop_Ap}
\left(\frac{\mu(E)}{\mu(R)}\right)^p\leq[w]_{A^*_p}\frac{w(E)}{w(R)}.
\end{equation}
For a given rectangle $R$, define
\[
E_R=\{x\in R: w(x)\leq\frac{1}{2^{p-1}[w]_{A^*_p}}w_R\}.
\] 
Hence, since $E_R$ is a $\mu$-measurable subset of $R$, \eqref{eq:prop_Ap} gives
\[
\left(\frac{\mu(E_R)}{\mu(R)}\right)^p\leq[w]_{A^*_p}\frac{w(E_R)}{w(R)}\leq[w]_{A^*_p}\frac{w_R}{w(R)}\mu(E_R)\frac{1}{2^{p-1}[w]_{A^*_p}}= \frac{1}{2^{p-1}}\frac{\mu(E_R)}{\mu(R)}.
\]
Then,
\begin{equation}\label{eq:prop_1}
\mu(E_R)\leq \frac{1}{2}\mu(R).
\end{equation}
Now we apply Lemma \ref{lem:MultiRiesz} to the rectangle $R$ to obtain a countable set of pairwise disjoint rectangles $R_j\in R$ satisfying 
\[
\avgint_{R_j}w \ d\mu=\lambda
\]
for each $j$,  and $w(x)\leq\lambda$ for $\mu$-a.e. points $x\in R\backslash\left(\bigcup_{j\geq 1}R_j\right)$. This decomposition together with \eqref{eq:prop_1} yields
\begin{eqnarray*}
w(\{x\in R:w(x)>\lambda \})&&\leq w(\bigcup_{j\geq 1}R_j)\leq\sum_j w(R_j)=\lambda \sum_j\mu(R_j)\\
&&\leq2\lambda \sum_j \mu(\{x\in R_j: w(x)>\frac{1}{2^{p-1}[w]_{A^*_p}}w_{R_j}\})\\
&&\leq2\lambda \mu(\{x\in R: w(x)>\frac{1}{2^{p-1}[w]_{A^*_p}}\lambda \}),
\end{eqnarray*}
since $w_{R_j}=\lambda$. Since $\lambda>w_R$, this yields \eqref{eq:prop_2}.
  
\end{proof}

We now present the proof of the dimension-free RHI for $A^*_p$ weights.

\begin{proof}[Proof of Theorem \ref{thm:RHI-dim-free}]
Define $\Omega_\lambda:=\{x\in R: w(x)>\lambda\}$. Then for arbitrary positive $\varepsilon$ we have
\begin{eqnarray*}
\avgint_R w(x)^{\varepsilon}w(x) \ d\mu & = &\frac{\varepsilon}{\mu(R)}\int_0^{\infty}\lambda^\varepsilon w(\Omega_\lambda)\frac{d\lambda}{\lambda }\\
&=&\frac{\varepsilon}{\mu(R)}\int_0^{w_R} \lambda^\varepsilon w(\Omega_\lambda) \ \frac{d\lambda}{\lambda } +
\frac{\varepsilon}{\mu(R)}\int_{w_R}^{\infty}\lambda^\varepsilon w(\Omega_\lambda) \ \frac{d\lambda}{\lambda }\\
&=& I+II.
\end{eqnarray*}
Observe that $I\leq(w_R)^{\varepsilon+1}$. To estimate $II$, we use Lemma \ref{lem:weak11}
\begin{eqnarray*}
II&= &\frac{\varepsilon}{\mu(R)}\int_{w_R}^{\infty}\lambda^{\varepsilon}w(\Omega_\lambda)\frac{d\lambda}{\lambda}\\
&\leq&\frac{2\varepsilon}{\mu(R)}\int_{w_R}^{\infty}\lambda ^{1+\varepsilon}\mu(\{x\in R:w(x)>\frac{1}{2^{p-1}[w]_{A^*_p}}\lambda \})\frac{d\lambda }{\lambda }\\
&=&(2^{p-1}[w]_{A^*_p})^{1+\varepsilon}\frac{2\varepsilon}{\mu(R)}\int_{\frac{w_R}{2^{p-1}[w]_{A^*_p}}}^{\infty}\lambda ^{\varepsilon+1}\mu(\Omega_\lambda)\frac{d\lambda }{\lambda }\\
&\leq&(2^{p-1}[w]_{A^*_p})^{1+\varepsilon}2\frac{\varepsilon}{1+\varepsilon}\avgint_Rw^{1+\varepsilon}\ d\mu.
\end{eqnarray*}
Setting $0<\varepsilon\leq\frac{1}{2^{p+2}[w]_{A^*_p}}$, we obtain
\begin{equation*}
II\leq\frac{1}{2}\avgint_Rw^{1+\varepsilon}\ d\mu.
\end{equation*}
where we have used that $t^{1/t}\leq 2$ whenever $t\geq1$. Therefore we obtain 
\[
\avgint_R w^{1+\varepsilon}\ d\mu\leq 2\left(\avgint_R w \ d\mu\right)^{1+\varepsilon},
\]
which is the desired estimate. 

\end{proof}

\begin{remark}
Clearly, Lemma \ref{lem:weak11} above does not hold for $A^*_1$ weights. But any $A^*_1$ weight $w$ can be viewed as an $A^*_p$ weight for any $p>1$. Therefore we have that $w$ satisfies a RHI for any exponent $\varepsilon$ such that $0<\varepsilon<\frac{1}{2^p[w]_{A^*_p}}$. Since the quantity $[w]_{A^*_p}$ increases to $[w]_{A^*_1}$, we conclude that the same result of Theorem \ref{thm:RHI-dim-free} is valid for $A^*_1$ weights with $0<\varepsilon<\frac{1}{2^{\eta}[w]_{A^*_1}}$ for any $\eta>3$.
\end{remark}

\section{Full range of local integrability for strong \texorpdfstring{$A_1^*$}{A1} weights}\label{sec:full-range}

In this section we show how to apply Lemma \ref{lem:MultiRiesz} to prove the full range of local integrability for $A_1^*$ weights. The key is to obtain a sort of self-improving property for the operator $M_s$ defined in \eqref{eq:M-strong-mu}.

\begin{proof}[Proof of Theorem \ref{thm:full-range-A1}:]
Set $\Omega_t:=\{x\in R: M_sw(x)\ge t\}$. Then for any arbitrary positive $\varepsilon$ we have
\begin{eqnarray*}
\avgint_R (M_sw)^{\varepsilon}w\ dx & \le &\frac{\varepsilon}{\mu(R)}\int_0^{\infty}t^{\varepsilon -1}w(\Omega_t)\ dt\\
&=&\frac{\varepsilon}{\mu(R)}\int_0^{w_R}t^{\varepsilon -1}w(\Omega_t)\ dt +
\frac{\varepsilon}{\mu(R)}\int_{w_R}^{\infty}t^{\varepsilon -1}w(\Omega_t)\ dt\\
&\le & (w_R)^{\varepsilon+1}+ \frac{\varepsilon}{\mu(R)}\int_{w_R}^{\infty}t^{\varepsilon -1}w(\Omega_t)\ dt.
\end{eqnarray*}
To estimate the last integral, we use Lemma \ref{lem:MultiRiesz} to obtain a collection of disjoint rectangles $\{R_j\}$ contained in $R$ such that 
\begin{equation*}
 \avgint_{R_j}w\ dx = t \qquad \text{ and } \qquad w(x)\le t \qquad \text{ for a.e. } x\in R\setminus \cup_j R_j.
\end{equation*}
Set $E:=R\setminus \bigcup_j R_j$. Then, 
\begin{eqnarray*}
 w(\Omega_t) & = & w\left(\Omega_t\cap \cup_j R_j\right) +  w\left(\Omega_t\cap E\right)\\
& \le & t \sum_j |R_j| + t \left|\Omega_t\cap E\right|.
\end{eqnarray*}

Note that for any $x\in \cup_j R_j$, we have that $M_sw(x)\ge t$, and therefore we obtain 
$\sum_j |R_j|\le |\Omega_t\cap (\cup_j R_j)|$. Hence,
\begin{eqnarray*}
w(\Omega_t) &\le & t|\Omega_t\cap (\cup_j R_j)|+t \left|\Omega_t\cap E\right|\\
&\le & t|\Omega_t|.
\end{eqnarray*}

And then
\begin{eqnarray*}
\frac{\varepsilon}{\mu(R)}\int_{w_R}^{\infty}t^{\varepsilon -1}w(\Omega_t)\ dt&\le&\frac{\varepsilon }{\mu(R)}\int_{w_R}^{\infty}t^{\varepsilon}|\Omega_t|\ dt \\
&\le&\frac{\varepsilon }{1+\varepsilon}\avgint_R (M_sw)^{1+\varepsilon}\ d\mu\\
&\le&\frac{\varepsilon [w]_{A^*_1} }{1+\varepsilon}\avgint_R (M_sw)^{\varepsilon}w\ d\mu.
\end{eqnarray*}
Collecting all estimates, we have that 
\begin{equation}\label{eq:maximal-self-A1}
\avgint_R (M_sw)^{\varepsilon}w\ d\mu  \le    (w_R)^{\varepsilon+1} +
\frac{\varepsilon [w]_{A^*_1} }{1+\varepsilon}\avgint_R (M_sw)^{\varepsilon}w\ d\mu.
\end{equation}
Setting $0<\varepsilon < \frac{1}{[w]_{A^*_1}-1}$, \eqref{eq:maximal-self-A1} yields
\begin{equation*}
\avgint_R (M_sw)^{\varepsilon}w \ d\mu \le \frac{1+\varepsilon}{1-\varepsilon([w]_{A^*_1}-1)}\left(\avgint_R w\ d\mu\right)^{1+\varepsilon}.
\end{equation*}
To finish, we take $1<s<\frac{[w]_{A^*_1}}{[w]_{A^*_1}-1}$ and let $\varepsilon=s-1$. Then
\[
\avgint_R w^s \ d\mu\leq \avgint_R (M_sw)^{(s-1)}w\ d\mu \le \frac{s}{1-(s-1)([w]_{A^*_1}-1)}\left(\avgint_R w\ d\mu\right)^s,
\]
which is the desired estimate.
\end{proof}

\section{The case of \texorpdfstring{$A_\infty^*$}{Ainfty} weights}\label{sec:Ainfty}

Until now, we have been focused on $A_p^*$ weights with $1\le p<\infty$. The aim of this section is to investigate a quantitative reverse H\"older property for the $A^*_\infty$ class in terms of its constant. First, we remark here that in this case there are several possible definitions of $[w]_{A_\infty^*}$.  Apart from the natural definition \eqref{eq:A^*inf}, a classical definition of the $A_\infty^*$ constant is the one obtained by taking the limit in the $A_p^*$ condition:

\begin{equation}\label{eq:Ainfty-exp-Rn}
[w]^{exp}_{A^*_\infty}:=\sup_R \left(\frac{1}{\mu(R)}\int_{R} w\,d\mu\right) \exp \left(\frac{1}{\mu(R)}\int_{R} \log w^{-1}\,d\mu  \right) <\infty
\end{equation}
where the supremum is taken over all rectangles $R\in\mathbb{R}^n$ with sides parallel to the coordinate axes. See \cite{Hruscev} for more details on this definition. However, the current tendency is to use a different  $A_\infty$ constant (implicitly introduced  by Fujii in \cite{Fujii}), which seems to be better suited:
\begin{equation}\label{eq:Ainfty-Rn}
   [w]_{A^*_\infty}:=\sup_R\frac{1}{w(R)}\int_R M_s(w\chi_R )\ d\mu <\infty.
\end{equation}
If the measure $\mu$ is doubling, definitions \eqref{eq:Ainfty-exp-Rn},\eqref{eq:Ainfty-Rn} and \eqref{eq:A^*inf} define the same class of weights. However, for general measures some extra conditions need to be imposed to establish the equivalence. For further details in the case of $A_{\infty}$ weights,  see \cite[Remark 2.4]{OP-nondoubling}. 

Below, we consider separately the cases of dimension $n=1$ and $n>1$. 

\subsection{\texorpdfstring{$A_\infty$ for the line}{Ainfty-R}}\label{sec:Ainfty-R}
In this case clearly there is no difference between cubic and rectangular weights and both definitions are equivalent when $\mu$ is doubling. It makes sense also in this one-dimensional case to use the centered maximal function $M^c$ instead of $M$ in definition \eqref{eq:Ainfty-Rn}:
\begin{equation*}
   [w]^c_{A_\infty}:=\sup_I\frac{1}{w(I)}\int_I M^c(w\chi_I )\ d\mu.
\end{equation*}
Note that this other definition is again equivalent to the others whenever $\mu$ is doubling. However, as we  remark before, when the underlying measure $\mu$ is non doubling, the equivalence is not clear. It can be shown, as in \cite[Proposition 2.2]{HP}, that $[w]_{A_\infty}\le c_n [w]^{exp}_{A_\infty}$. This inequality relies on the fact that $M$ is bounded on $L^p(\mu)$ for any measure $\mu$. Also It is obvious that $[w]^c_{A_\infty}\le [w]_{A_\infty}$ but the finiteness of $[w]^c_{A_\infty}$ does not characterize $A_\infty$. Indeed in \cite[p. 2021]{OP-nondoubling} there is an example of a weight $w$ which is \emph{not} in $A_\infty$ satisfying that $M^cw\lesssim w$ for $\mu$-a.e. $x\in \mathbb{R}$. In other words, the centered maximal operator is too small to characterize $A_\infty$. 

The following result in this section shows that in fact $[w]_{A_\infty}$ characterizes $A_\infty$.

\begin{theorem}\label{thm:RHI-Ainfty-line}
 Let $\mu$ be any non atomic Radon measure on $\mathbb{R}$ and let $w$ be a weight such that $[w]_{A_\infty}<\infty$. Then it satisfies the following RHI. For any $0<\varepsilon < \frac{1}{4[w]_{A_\infty}-1}$ and for any interval $I$, we have that

\begin{equation}\label{eq:RHI-line}
\avgint_I w^{1+\varepsilon}\ d\mu \leq2\left(\avgint_I w\ d\mu\right)^{1+\varepsilon}. 
\end{equation}
\end{theorem}
\begin{remark} Using the characterization from \cite[Lemma 2.3]{OP-nondoubling}, we deduce from this theorem that $w\in A_\infty$.
\end{remark}

\begin{proof}

We use an specific stopping time argument adapted to the $\mu$-dyadic grid for a given interval $I$. We begin with a similar idea as in  \cite[p. 536]{MMNO}, where a proof of John-Nirenberg's inequality for non-atomic measures in the real line is presented. We sketch here the construction. The first generation $G_1(I)$ of the dyadic grid  consists of the two disjoint subintervals $I_+$, $I_-$ of $I$ satisfying $\mu(I_+) = \mu(I_-) = \mu(I)/2$. The second generation $G_2(I)$ is $G_1(I_+)\cup G_1(I_-)$. Next generations are defined recursively. Since the measure has no atoms, we can take closed intervals sharing the endpoints. Let $\mathcal{D}_{I}^\mu$ be the family of all the dyadic intervals generated with this procedure. 
A collection of nested intervals from this grid will be called a \emph{chain}. More precisely, a chain $\mathcal{C}$ will be of the form $\mathcal{C}=\{J_i\}_{i\in\mathbb N}$ such that $J_i\in G_i(I)$, and $J_{i+1}\subset J_i$ for all $i\ge1$.  

If we define $\mathcal{C}_\infty:=\bigcap_{J\in \mathcal{C}} J$ as the \emph{limit set} of the chain $\mathcal{C}$, we have that $\mathcal{C}_\infty$ 
could be a single point or a closed interval of positive length. In any case, we clearly have that $\mu(\mathcal{C}_\infty)=0$. 
We will say that  those limit sets $\mathcal{C}_\infty$ of positive length are \emph{removable}. Since we are in the real line, there are at most a countable many of them and the whole union is also a $\mu$-null set. We denote by $\mathcal{R}$ the set of all chains with removable limits.
If we define
\begin{equation}\label{eq:removed-line}
 E:= I\setminus \bigcup_{\mathcal{C}\in \mathcal{R}}\mathcal{C_\infty},
\end{equation}
we conclude that $\mu(I)=\mu(E)$ and, in addition, for any $x\in E$, there exists a chain of nested intervals shrinking to $x$. Therefore the grid $\mathcal{D}_{I}^\mu$ forms a differential basis on $E$. Moreover, the dyadic structure of the basis guarantees the Vitali covering property (see \cite[Ch.1]{guz75} ) and therefore this basis differentiates $L^1(E)$. 

Associated to this grid we define a \emph{dyadic} maximal operator as follows. For any $x\in E$,
\begin{equation*}
M^{\mathcal{D}_{I}^\mu }f(x) 
=
\sup_{J\in \mathcal D_I^\mu}\avgint_J |f|\ d\mu,
\end{equation*} 
By a standard differentiation argument, we have that this maximal function satisfies that $f\le M^{\mathcal{D}_{I}^\mu } f$, $f \geq 0$,  almost everywhere on $E$. 

Now the proof of the main inequality \eqref{eq:RHI-line} follows the same steps as in \cite[Lemma 2.2]{HPR1}. First, we prove the following inequality for the maximal operator. We claim that, for any 
$0<\varepsilon \le\frac{1}{4[w]_{A_\infty}-1}$,
we have that
\begin{equation}\label{eq:RHIMaximal-line}
 \avgint_{I} (M^{\mathcal{D}_{I}^\mu }(\chi_{I}w))^{1+\varepsilon}\ d\mu\le 2[w]_{A_\infty}\left(\avgint_{I} w\ d\mu\right)^{1+\varepsilon}.
\end{equation} 
To simplify the notation throughout the proof of this inequality, we will denote $w:=w\chi_{I}$, $M:=M^{\mathcal{D}_{I}^\mu }$  and $\Omega_\lambda:=I\cap\{Mw>\lambda \} $. 
We start with the following identity:
\begin{equation*}
\int_{I} (Mw)^{1+\varepsilon}\ d\mu\leq  \int_0^{w_I} \varepsilon \lambda^{\varepsilon-1}\int_{I}Mw d\mu\ d\lambda + \int_{ w_I}^\infty \varepsilon \lambda^{\varepsilon-1}Mw(\Omega_\lambda )\ d\lambda.
\end{equation*}
Now, for $\lambda\ge w_I$, there is a family of maximal nonoverlapping $\mu$-dyadic intervals $\{I_j\}_j$ for which 
\begin{equation*}
\Omega_\lambda=\bigcup_j I_j \quad \text{ and }  \quad \avgint_{I_j}w\ d\mu >\lambda.
\end{equation*}
Therefore, by using this decomposition and the definition of the  $A_\infty$ constant, we can write
\begin{equation}\label{eq:sum0}
\int_I (Mw)^{1+\varepsilon}\ d\mu
 \le  w_I^\varepsilon [w]_{A_\infty}w(I) +\int_{w_I}^\infty \varepsilon \lambda^{\varepsilon-1} \sum_j \int_{I_j} Mw\ d\mu d\lambda.
\end{equation}
By maximality of the intervals in $\{I_j\}_j$, it follows that the dyadic maximal function $M$ can be localized:
\begin{equation*}
 Mw(x)=M(w\chi_{I_j})(x),
\end{equation*}
for any $x\in I_j$, for all $j\in\mathbb{N}$. Now, if we denote by $\widetilde{I}$  the dyadic parent of a given interval $I$, then we have that
\begin{equation*}
 \int _{I_j}M(w\chi_{I_j})d\mu \le [w]_{A_\infty}w(I_j) \le [w]_{A_\infty}w_{\widetilde{I_j}} \mu(\widetilde{I_j})\le [w]_{A_\infty}\lambda 2\mu(I_j).
\end{equation*}
Therefore, after averaging over $I$, we have that \eqref{eq:sum0} becomes
\begin{equation*}
\avgint_{I}(Mw)^{1+\varepsilon}\ d\mu \le   w_{I}^{1+\varepsilon} [w]_{A_\infty} +
\frac{\varepsilon 2[w]_{A_\infty}}{1+\varepsilon}\avgint_I (Mw)^{1+\varepsilon}\ d\mu.
\end{equation*}
We conclude with the proof of inequality \eqref{eq:RHIMaximal-line} by absorbing the last term into the left, since $0<\varepsilon\le \frac{1}{4[w]_{A_\infty}-1}$.

Now we argue in a similar way to obtain, by using that $w\le Mw$, the following estimate

\begin{equation*}
 \int_{I}w^{1+\varepsilon}\ d\mu \le \int_0^\infty \varepsilon \lambda^{\varepsilon-1}w(\Omega_\lambda )\ d\lambda 
\le w_{I}^\varepsilon w(I) + \int_{w_{I}}^\infty \varepsilon \lambda^{\varepsilon-1} \sum_j w(I_j)\ d\lambda,
\end{equation*}
where the cubes $\{I_j\}_j$ are from the decomposition of $\Omega_\lambda$ above. Therefore, using again that $w(I_j)\le 2\lambda\mu(I_j)$, we get 
\begin{eqnarray*}
\int_{I} w^{1+\varepsilon}\ d\mu &\le &w_{I}^\varepsilon w(I) + 2\varepsilon \int_{w_{I}}^\infty \lambda^{\varepsilon} \mu(\Omega_\lambda)\ d\lambda\\
 &\le &w_{I}^\varepsilon w(I) +  \frac{2\varepsilon}{1+\varepsilon}\int_{I} (Mw)^{1+\varepsilon}\ d\mu.
\end{eqnarray*}
Averaging over $I$ and using \eqref{eq:RHIMaximal-line}  we obtain

\begin{eqnarray*}
\avgint_{I} w^{1+\varepsilon}\ d\mu 
&\le & w_{I}^{1+\varepsilon}+ 
\frac{4\varepsilon [w]_{A_\infty}}{1+\varepsilon}\left(\avgint_{I} w\ d\mu\right)^{1+\varepsilon}\\
&\le & 2 \left(\avgint_{I} w\ dx\right)^{1+\varepsilon},
\end{eqnarray*}
where in the last step we have used that $\frac{\varepsilon 2[w]_{A_\infty}}{1+\varepsilon}\le \frac{1}{2}$.

\end{proof}

There are two immediate consequences of this result. Firstly, we have the following precise open property for one-dimensional $A_p$ weights (compare this to \eqref{eq:Ap-Ap-e-maximal}).

\begin{corollary}\label{cor:Ap-Ap-e-Ainfty}  Let $\mu$ be any non atomic measure on $\mathbb{R}$. For $1<p<\infty$ and $w\in
A_p$, define the quantity $r(w)=1+\frac{1}{4[w]_{A_{\infty}}}$. Then  $w\in A_{p-\varepsilon} $ where 
\begin{equation*}
\varepsilon =\frac{p-1}{r(\sigma)' }= \frac{p-1}{ 1+4[\sigma]_{A_{\infty}} } 
\end{equation*}
and $\sigma=w^{1-p'}$. Furthermore, $[w]_{A_{p-\varepsilon}} \le  2^{p-1}[w]_{A_p}.$

\end{corollary}

We omit the proof of this corollary because, since it does not depend on further properties of the measure, it is exactly the same as in \cite{HPR1}. 

The next corollary is a mixed $A_p$--$A_\infty$ estimate for the H--L maximal operator $M$. The result for spaces of homogeneous type can be found in \cite{HPR1}. Further improvements based on a different approach avoiding the RHI property has been obtained in \cite{PR-twoweight}.

\begin{corollary}\label{cor:SharpBuckley}   Let $\mu$ be any non atomic Radon measure on $\mathbb{R}$ and let $M$ be the Hardy-Littlewood maximal function. For  $1<p<\infty$ and $w\in A_p$, define  as above $\sigma=w^{1-p'}$. Then there is a constant $C>0$ such that
\begin{equation*}
  \|M\|_{L^p( w)} \leq  c\,  \left( p' [w]_{A_p}[\sigma]_{A_\infty}\right)^{1/p}.
\end{equation*}
\end{corollary}

Recall that as in the rest of the paper $\|M\|_{L^p( w)}$ is the $L^p$ operator norm of $M$ with respect to $wd\mu$. 

\begin{proof}
For the proof of the corollary we need the following weak weighted norm estimate for the maximal function. 
\begin{equation}\label{eq:Maximal-weak}
 \|M\|_{L^{q,\infty}(w)} \le 5\,[w]_{A_q}^{\frac1q},\qquad      \qquad 1<q<\infty.
\end{equation} 
Consider, for any nonnegative measurable function $f$ and $\lambda>0$, the level set $\Omega_\lambda=\{x\in \mathcal{\mathbb R}: Mf(x)>\lambda\}$. 
Since we are in the real line, we can proceed by using a covering lemma specific for 1 dimensional intervals (see the details in \cite{sjogren}, p. 1232). We can obtain a countable family of disjoint intervals $\{I_j\}_j$ such that 
\begin{equation*}
\frac{1}{\mu(I_j)}\int_{I_j}f\ d\mu >\lambda \qquad \mbox{ and }\quad 
\Omega_\lambda \subset \bigcup_j I_j^*
\end{equation*}
where $I_j\subset I_j^*$ and $\mu(I_j^*)\le5\mu(I_j)$. Therefore
\begin{eqnarray*}
 \lambda^qw(\Omega_\lambda) & \le & \sum_j w(I_j^*)\left(\frac{1}{\mu(I_j)}\int_{I_j}fw^\frac{1}{q}w^{-\frac{1}{q}}\ d\mu\right)^q\\
& \le & 5^q\sum_j \frac{w(I_j^*)}{\mu(I^*_j)}\left(\frac{1}{\mu(I^*_j)}\int_{I^*_j}\sigma\ d\mu\right)^{q-1}
\int_{I_j}f^qw\ d\mu\\
&\le &   5^q[w]_{A_q}\|f\|^q_{L^q(w)}
\end{eqnarray*}
and then \eqref{eq:Maximal-weak} follows. The next steps are the same as in \cite[Theorem 1.3]{HPR1}; we sketch the proof for completeness. Indeed, by a change of variables an using the above relation between the level sets, we write 
\begin{equation*}
\|Mf\|_{L^p(w)}^p \leq p 2^p \int_{0}^{\infty}  t^{p} w \{y\in \mathbb R:M(f\chi_{f>t})(y) > t\}
  \frac{dt}{t}.
\end{equation*}
Using the weak norm estimate for $A_{p-\varepsilon}$ \eqref{eq:Maximal-weak}, we obtain
\begin{equation*}
 \|Mf\|_{L^p(w)}^p \leq p10^p \frac{[w]_{A_{p}}}{\varepsilon} \int_{\mathbb R}    f^{p}  w\ d\mu.
\end{equation*}
The desired inequality follows chosing  $\varepsilon=\frac{p-1}{ 1+4[\sigma]_{A_{\infty}} }$, from Corollary \ref{cor:Ap-Ap-e-Ainfty}. 

\end{proof}

\subsection{Higher dimensions: \texorpdfstring{$A^*_\infty$ for $\mathbb{R}^n$}{Ainfty-Rn}}\label{sec:Ainfty-Rn}\

The first observation  is that in higher dimensions and for any doubling measure we can easily adapt the result from \cite{HPR1} to strong weights.
\begin{theorem}\label{thm:Ainfty-RHI-doubling}
Let $\mu$ be a doubling measure on $\mathbb{R}^n$ and let $w\in A^*_\infty$. Then for any rectangle $R$,
\begin{equation*}
\avgint_{R} w^{1+\varepsilon}\ d\mu  \le 2\left(\avgint_{R} w \ d\mu\right)^{1+\varepsilon},
\end{equation*}
for any $\varepsilon>0$  such that $0<\varepsilon \le\frac{1}{2C_\mu[w]_{A^*_\infty}-1 }$. Here $C_\mu$ depends on the doubling constant of the measure.
\end{theorem}

The key is to consider a local dyadic version of the maximal operator. For a fixed rectangle $R_0$, we also consider  the dyadic local maximal operator $M_{R_0}^d$  defined by averages over dyadic children of $R_0$. More precisely, we consider $R_0$ as a part of a dyadic grid and  the family $\mathcal{D}(R_0)$ will be obtained by successive dyadic subdivisions of the rectangle $R_0$. Then, define
\begin{equation*}
 M_{R_0}^d f(x)=\sup_{R\ni x, R\in \mathcal{D}(R_0)} \avgint_{R} |f|\ d\mu.
\end{equation*}

Now, for general measures, we can track carefully the constants through the proof from \cite{OP-nondoubling} and obtain the analogue of Theorem \ref{thm:Ainfty-RHI-doubling} for cubes and the constant $[w]^{exp}_{A_\infty}$. More precisely, it can be proved that for any $w\in A_\infty$, there exists a constant $C$ such that  
\begin{equation*}
\avgint_{Q} w^{1+\varepsilon}\ d\mu \le C\left(\avgint_{Q} w \ d\mu\right)^{1+\varepsilon}
\end{equation*}
holds whenever
$$0<\varepsilon \le \frac{c_n}{2[w]^{exp}_{A_\infty}\left(e^{2[w]^{exp}_{A_\infty}}-1 \right)}.
$$
Here $c_n$ denotes a dimensional constant. Clearly, this range for $\varepsilon$ is worse than the one obtained in Theorem \ref{thm:Ainfty-RHI-doubling} or Theorem \ref{thm:RHI-Ainfty-line}.

The main result of this section involves the (strong) Fujii-Wilson constant $[w]_{A^*_\infty}$ defined in \eqref{eq:Ainfty-Rn}. A first problem that we can consider is to determine the validity of the inequality $[w]_{A^*_\infty}\le c_n [w]_{A^*_\infty}^{exp}$.
This estimate in dimension 1 is consequence of the $L^p(\mu)$ boundedness of $M$ which is always true for any measure $\mu$.  However, the corresponding question in higher dimensions is still open since it could be the case that the maximal function $M$ is bounded only on $L^{\infty}$, and not in any $L^p$, $1<p<\infty$ even for the centered case (see for example \cite{sjogren-soria} for recent developments on this subject).  

Going back to rectangles, we are able to describe a rather abstract theorem for strong weights than Theorem \ref{thm:RHI-Ainfty-line}. The general standing assumption on the measure $\mu$ will be the absence of atoms. As we already mentioned, we can then assume that the measure of hyperplanes parallel to the coordinates axes is zero. Therefore, we can define the same $\mu$-dyadic grid by splitting any fixed rectangle $R$ into $2^n$ sub-rectangles $\{R_i:1\le i \le 2^n\}$ such that $\mu(R_i)=2^{-n}\mu(R)$ (note that there is not a unique way of doing this).
Start with a given rectangle $R_0$ and define recursively the the dyadic grid $\mathcal{D}_{R_0}^\mu$. As before, the corresponding local maximal operator is
\begin{equation*}
M^{\mathcal{D}_{R_0}^\mu}f(x)=\sup_{R\in \mathcal{D}_{R_0}^\mu} \avgint_R |f|\ d\mu.
\end{equation*}

We have the following theorem. Since the proof follows the same steps as in the one dimensional case, we left details to the reader. 

\begin{theorem}\label{thm:Ainfty-RHI-RN}
Let $\mu$ be a non-atomic Radon measure on $\mathbb{R}^n$ and let $w\in A^*_\infty$. Suppose, in addition, that there is a constant $C$ such that, for any rectangle $R$, $w(x)\le CM^{\mathcal{D}_{R}^\mu}w(x)$ $\mu$-a.e on $R$.  Then for any rectangle $R$,
\begin{equation*}
\avgint_{R} w^{1+\varepsilon}\ dx \le 2C\left(\avgint_{R} w \ dx\right)^{1+\varepsilon},
\end{equation*}
for any $\varepsilon>0$  such that $0<\varepsilon \le\frac{1}{2^{n+1}[w]_{A^*_\infty}-1 }$.  
\end{theorem}

It would be interesting to characterize those measures fulfilling the hypothesis of the above theorem. First, in order to have a new and better estimate, we need the HL maximal operator to be bounded on $L^p(\mu)$ for some $p<\infty$. In addition, although it could seem trivial, it is not always true that the local maximal operator defined in terms of the dyadic grid majorizes the function. This will depend on the geometry of the grid. 

\begin{example}
A family of measures for which we can solve both problems is the family of tensor product measures. A classical example of a nondoubling measure of this type is the Gaussian measure $\mu_\delta$ with density $d\mu_\delta(x)=e^{-|x|^\delta}dx$. 
We will assume that the measure $\mu$ on $\mathbb{R}^n$ can be written as $\mu=\bigotimes_{i=1}^n \mu_i$, where $\mu_1,\mu_2,\dots, \mu_n$ are defined on $\mathbb{R}$ and none of them has atoms. 
In this case, by iterating the result for the real line, we know that the HL maximal function over rectangles (and, a fortiori, over cubes) is bounded on $L^p(\mu)$ and therefore the constant $[w]_{A^*_\infty}$ provides better estimates. To verify that $w \leq M^{\mathcal{D}_{R}^\mu}(w) $ a.e. $\mu$, we need to perform the dyadic partition on each direction separately. Suppose that the rectangle $R$ is of the form $R=\prod_{i=1}^n I_i$. We perform the partition on each direction to obtain the dyadic grid $\mathcal D^\mu_i=\bigcup_{j\ge 1}G_j(I_i)$. Following the same idea as in the linear case, we call $\mathcal{R}_i$ the family of all chains with \emph{removable} limits on each direction. After removing all of them, we can assume that any chain $\mathcal{C}=\{J_m\}_{m\in \mathbb{N}}$ in $\mathcal{D}^\mu_i$ verifies that $\lim_{m\to \infty}\text{diam}(J_m)=0$. Define in a similar way as in \eqref{eq:removed-line} the sets
\begin{equation*}
 E_i:= I_i\setminus \bigcup_{\mathcal{C}\in \mathcal{R}_i}\mathcal{C_\infty},\qquad 1\le i\le n
\end{equation*}
and 
\begin{equation*}
 E:=E_1\times\cdots\times E_n.
\end{equation*}

We can build the dyadic grid for $R$ taking the products elements of each $\mathcal{D}^\mu_i$ of the same level. More precisely, the $k$-th level dyadic grid is 
\begin{equation*}
 \mathcal{D}_k^\mu=\left\{R=J_1\times \cdots \times J_n: J_i\in G_k(I_i), 1\le i\le n \right\},
\end{equation*}
and the complete grid is the union of all levels:
\begin{equation*}
 \mathcal{D}^\mu=\bigcup_k \mathcal{D}_k^\mu.
\end{equation*}

The grid $\mathcal{D}_k^\mu$ defined in this way is a differential basis on $E$ as in the 1-dimensional case satisfying the Vitali covering property. Hence,  the same reasoning used before allows us to conclude that $w\le M_s^{d_\mu}w$ for $\mu$-almost all $x$ in $E$.

\end{example}

As a remark related to Theorem \ref{thm:Ainfty-RHI-RN} and the above example, we can derive a result in the spirit of Corollary \ref{cor:SharpBuckley} for the strong maximal function $M_s$ associated to a $n$-product of non-atomic Radon measures on $\mathbb{R}$   $\mu=\bigotimes_{i}^n \mu_i$ be any $n$-product of non-atomic Radon measures on $\mathbb{R}$: 
\begin{equation}\label{eq:StrongHLLpAinfty}
  \|M_s\|_{L^p(wd \mu)} \leq c\, (p')^n  [w]^{\frac{1}{p}+ 2\frac{n-1}{p-1}}_{A^*_p}[\sigma]_{A^*_\infty}^{\frac{1}{p}}   \qquad 1<p<\infty 
\end{equation}
where as usual $\sigma=w^{1-p'}$. Unfortunately this result is far from being sharp since it can be shown that the expected consequence, namely 
\begin{equation}\label{Lp-pEstimate}
\|M_s\|_{L^p(wd \mu)} \leq c\, (p')^n  [w]^{\frac{n}{p-1}}_{A^*_p} 
\end{equation}
cannot be derived. For this reason we omit the proof of \eqref{eq:StrongHLLpAinfty} which is based on similar arguments as before, namely combining an appropriate  weak norm estimate 
\begin{equation}\label{WeakLp-p-Estimate}
\|M_s\|_{L^q(wd\mu)\to L^{q,\infty}(wd\mu)} \le C_n(q')^{n-1}[w]_{A^*_q}^{\frac1q+\frac{n-1}{q-1}}     \qquad 1<q<\infty,
\end{equation} 
together with the open property derived from the RHI property from Theorem \ref{thm:Ainfty-RHI-RN}.

\section{ Further variants of RHI for \texorpdfstring{$A_p^*$}{Ap} weights}
 \label{sec:further}

We include here some additional versions of the RHI related to the operator norm of the maximal function. For the sake of clarity, in this section we restrict ourselves to the case of the Lebesgue measure, although some of the results are valid in a wider scenario.

Let us recall that Buckley's result from \cite{Buckley} states that the maximal function $M$ (over cubes) satisfies
\begin{equation}\label{eq:buckley}
\|M\|_{L^p(w)}\le c_p [w]_{A_p}^{\frac1{p-1}} , \qquad w\in A_p(dx)
\end{equation}

This result can be obtained using the following result from \cite{LO}: if $w\in A_p$, $p>1$. Then for each cube $Q$
\begin{equation*}
\avgint_Q w^{1+\varepsilon}\ dx\leq2\left(\avgint_Q w\ dx\right)^{1+\varepsilon}.
\end{equation*}
for any $0<\varepsilon\leq   \frac{1}{2^{n+2}  \|M_s\|_{L^{p'}(\sigma)} }, $ where $\sigma=w^{1-p'}$. This results implies the open property, namely if  $w\in A_p$ implies $w \in A_{p-\varepsilon}$ with  $\varepsilon= \frac{p-1}{ 1+  2^{n+2}  \|M\|_{L^{p}(w)}   } $\, and \,$[w]_{A_{p-\varepsilon}} \le  2^{p-1}[w]_{A_p}.$  Now, using the same argument as in the proof Corollary  \ref{cor:SharpBuckley}  based on  \cite[Theorem 1.3]{HPR1}, it follows easily the following result 
\begin{equation}\label{eq:weak-strong1}
\|M\|_{L^p( w)} \leq c_np'\,  \|M\|_{ L^p( w) \to L^{p,\infty}( w) }\, \|M\|_{L^p( w)}^{1/p} \qquad w \in A_p.
\end{equation}
This gives another proof of \eqref{eq:buckley} although the constant $c_p$ is not the correct one, namely 
$c_p \approx p'^{p'}$ instead of $c_p \approx p'$.

Concerning the strong maximal function $M_s$, a still not answer question is whether \eqref{eq:buckley} holds with the same exponent or not.

We do not know if \eqref{eq:weak-strong1} holds for the strong maximal function. However, if this estimate were true it seems that it is not of so much interest since the scheme just sketched breaks down. Indeed, the combination of  \eqref{eq:weak-strong1}  and \eqref{WeakLp-p-Estimate} does not lead to the expected result \eqref{Lp-pEstimate}. Of course this is related to the question of the precise dependence of the weighted norm $\|M_s\|_{ L^p( w) \to L^{p,\infty}( w) }$ in terms of $[w]_{A_p^*}$ which is still an open problem. 
In particular, for product weights there are sharp estimates for both the weak and strong norms (see inequality \eqref{weakp-pProductsweighs} below, details can be found in the forthcoming paper \cite{LP})
If we combine the sharp estimate for the weak norm \eqref{weakp-pProductsweighs} for product weights together with \eqref{eq:weak-strong1} we still do not recover the known sharp estimate for the strong norm for product weights. Anyway, we can prove a similar open property as in the cubic case as a consequence of the first part of the following result.

\begin{theorem}\label{thm:RHAp-maximal} 

Let $w\in A^*_p$, $p>1$. Then for each rectangle $R$
\begin{equation}\label{eq:RHI-epsilon-maximal}
\avgint_R w^{1+\varepsilon}\ dx\leq2\left(\avgint_R w\ dx\right)^{1+\varepsilon}
\end{equation}
for any $0<\varepsilon\leq   \frac{1}{2  \|M_s\|_{L^{p'}(\sigma)} }, $ where $\sigma=w^{1-p'}$.

Similarly, 
\begin{equation}\label{eq:Integrability-epsilon-maximal}
\avgint_R w^s \ d\mu \leq \frac{s}{1-(s-1)( \|M_s\|_{L^{p'}(\sigma)} -1)} \left(\avgint_R w\ d\mu\right)^s
\end{equation}
for any $1<s<\frac{  \|M_s\|_{L^{p'}(\sigma)} }{ \|M_s\|_{L^{p'}(\sigma)}-1}.$

\end{theorem}

\begin{proof}

It follows from the same argument as in the proof of Theorem \ref{thm:full-range-A1}, inequality \eqref{eq:maximal-self-A1}, that for any arbitrary positive $\varepsilon$ we have

\begin{equation*}
\avgint_R (M_sw)^{\varepsilon}w\ dx  \le    (w_R)^{\varepsilon+1} +
\frac{\varepsilon  }{1+\varepsilon}\avgint_R (M_sw)^{\varepsilon+1}dx.
\end{equation*}
Following \cite{LO}, by H\"older's inequality and the trivial bound
$w\le M_sw$ we obtain
\begin{eqnarray*}
\int_R (M_sw)^{\varepsilon+1}\,dx & = & \int_R(M_sw)^{\frac{\varepsilon}{p}} \,w^{1/p}\, (M_sw)^{1+\frac{\varepsilon}{p'}}\,w^{-1/p}\,dx\\
& \leq & \left(\int_R(M_sw)^{\varepsilon} \,wdx\right)^{1/p}
\left(\int_R  (M_sw)^{p'+\varepsilon }\,w^{1-p'}\,dx\right)^{1/p'}\\
& \leq &
\|M_s\|_{L^{p'}(\sigma)} 
\int_R(M_sw)^{\varepsilon} \,wdx\\
\end{eqnarray*}
In the last inequality we use that
\begin{equation*}
\|M_s\|_{L^{p_1}(\mu)}^{p_1}\leq \|M_s\|_{L^{p_2}(\mu)}^{p_2} \qquad \text{for} \qquad p_1 \geq p_2.
\end{equation*}

Hence, 
\begin{equation}\label{eq:maximal-self--further}
\frac{1}{|R|}\int_R (M_sw)^{\varepsilon}wdx\leq
(w_R)^{\varepsilon+1}
+
\frac{\varepsilon}{\varepsilon+1} \|M_s\|_{L^{p'}(\sigma)}
\frac{1}{|R|}\int_R(M_sw)^{\varepsilon} \,wdx
\end{equation}
and letting  \,$\varepsilon=\frac{1}{2  \|M\|_{L^{p'}(\sigma)} }$\, we get
$$
\frac{1}{|R|}\int_R (M_sw)^{\varepsilon}wdx\leq 2\,
(w_R)^{\varepsilon+1}
$$
This last estimate clearly yields \eqref{eq:RHI-epsilon-maximal}.

Simliarly, setting $0<\varepsilon < \frac{1}{\|M\|_{L^{p'}(\sigma)}-1}$, we have that
\begin{equation*}
\avgint_R (M_sw)^{\varepsilon}w \ d\mu \le \frac{1+\varepsilon}{1-\varepsilon(\|M\|_{L^{p'}(\sigma)}-1)}\left(\avgint_R w\ d\mu\right)^{1+\varepsilon},
\end{equation*}
which yields \eqref{eq:Integrability-epsilon-maximal}

\end{proof}

In a similar way as in Corollary \ref{cor:Ap-Ap-e-Ainfty}, we can derive also an alternative version of the open property for $A^*_p$ classes; more precisely, if $w \in A^*_p$ and  $r(w)=1+ \frac{1}{2  \|M_s\|_{L^{p}(w)} }$ then  $w\in A^*_{p-\varepsilon} $ where 
\begin{equation}\label{eq:Ap-Ap-e-maximal}
\varepsilon =\frac{p-1}{r(\sigma)' }= \frac{p-1}{ 1+2 \|M_s\|_{L^{p'}(\sigma)}  }. 
\end{equation}

Now we want to collect all the estimates for the RHI. We have Theorem \ref{thm:RHI-dim-free} and Theorem \ref{thm:RHAp-maximal}. In addition, we  also have Theorem \ref{thm:Ainfty-RHI-doubling} or Theorem \ref{thm:Ainfty-RHI-RN} since we are considering $\mu$ as the Lebesgue measure which is both doubling and a product measure.  Then, we have that any $w\in A^*_p$ satisfies a RHI with exponent $1+\varepsilon$ for any 

\begin{equation*}
0\le\varepsilon<\max \left\{ \frac{1}{2^{p+2}[w]_{A^*_p}}, \frac{1}{2\|M_s\|_{ L^{p'}(\sigma)}}, \frac{1}{2^{n+1}[w]_{A^*_\infty}-1}  \right\}
\end{equation*}

Which of these estimates is better, will depend on what is the best bound for the weighted norm of the strong maximal  function. It is clear that we have  $ \|M_s\|_{ L^{p'}(\sigma)} \leq  c_np^n [w]^n_{A^*_p}  $. But it is not known in general if the exponent on the constant of the weight can be smaller than $n$. 

Note for example that in the case of strong product weights, \cite[Theorem 3.7]{LP} shows that $2\|M_s\|_{L^{p'}(\sigma)}\leq 2^{p+2}[w]_{A^*_p}$ and therefore Theorem \ref{thm:RHAp-maximal} would provide a better result than Theorem \ref{thm:RHI-dim-free} in this case. Moreover in this particular case, \cite[Theorem 3.1]{LP} also assures the next sharp result:
\begin{equation} \label{weakp-pProductsweighs}
\|M_s f\|_{L^{p,\infty}(w)} \lesssim_{n,p}[w]_{A_p ^*} ^{\frac{1}{p-1}(1-\frac{1}{np})}\|f\|_{L^p(w)}.
\end{equation}

This combined with Theorem \ref{thm:RHAp-maximal} yield the multiparameter version of \eqref{eq:weak-strong1}. However,  we do not recover the sharp result in this particular case of product weights.

\bibliographystyle{plain}

\begin{thebibliography}{10}

\bibitem{AB}
Daniel Aalto and Lauri Berkovits.
\newblock Asymptotical stability of {M}uckenhoupt weights through
  {G}urov-{R}eshetnyak classes.
\newblock {\em Trans. Amer. Math. Soc.}, 364(12):6671--6687, 2012.

\bibitem{ACS}
Ravi~P. Agarwal, Shusen Ding, and Craig Nolder.
\newblock {\em Inequalities for differential forms}.
\newblock Springer, New York, 2009.

\bibitem{AIM}
Kari Astala, Tadeusz Iwaniec, and Gaven Martin.
\newblock {\em Elliptic partial differential equations and quasiconformal
  mappings in the plane}, volume~48 of {\em Princeton Mathematical Series}.
\newblock Princeton University Press, Princeton, NJ, 2009.

\bibitem{BSW}
B.~Bojarski, C.~Sbordone, and I.~Wik.
\newblock The {M}uckenhoupt class {$A_1({\mathbb{R}})$}.
\newblock {\em Studia Math.}, 101(2):155--163, 1992.

\bibitem{Buckley}
Stephen~M. Buckley.
\newblock Estimates for operator norms on weighted spaces and reverse {J}ensen
  inequalities.
\newblock {\em Trans. Amer. Math. Soc.}, 340(1):253--272, 1993.

\bibitem{CPP}
Daewon Chung, M.~Cristina Pereyra, and Carlos Perez.
\newblock Sharp bounds for general commutators on weighted {L}ebesgue spaces.
\newblock {\em Trans. Amer. Math. Soc.}, 364(3):1163--1177, 2012.

\bibitem{CF}
Ronald~R. Coifman and C.~Fefferman.
\newblock Weighted norm inequalities for maximal functions and singular
  integrals.
\newblock {\em Studia Math.}, 51:241--250, 1974.

\bibitem{guz75}
Miguel de~Guzm{\'a}n.
\newblock {\em Differentiation of integrals in {$R\sp{n}$}}.
\newblock Springer-Verlag, Berlin, 1975.
\newblock With appendices by Antonio C\'ordoba, and Robert Fefferman, and two
  by Roberto Moriy\'on, Lecture Notes in Mathematics, Vol. 481.

\bibitem{DMRO-Ainfty}
Javier Duoandikoetxea, Francisco Mart\'in-Reyes, and Sheldy Ombrosi.
\newblock On the {$A_\infty$} conditions for general bases.
\newblock 2013.
\newblock DOI: 10.1007/s00209-015-1572-y.

\bibitem{Fujii}
Nobuhiko Fujii.
\newblock Weighted bounded mean oscillation and singular integrals.
\newblock {\em Math. Japon.}, 22(5):529--534, 1977/78.

\bibitem{GCRdF}
Jos{\'e} Garc{\'{\i}}a-Cuerva and Jos{\'e}~L. Rubio~de Francia.
\newblock {\em Weighted norm inequalities and related topics}, volume 116 of
  {\em North-Holland Mathematics Studies}.
\newblock North-Holland Publishing Co., Amsterdam, 1985.

\bibitem{Gehring}
F.~W. Gehring.
\newblock The {$L^{p}$}-integrability of the partial derivatives of a
  quasiconformal mapping.
\newblock {\em Acta Math.}, 130:265--277, 1973.

\bibitem{HagLuPar}
Paul Hagelstein, Teresa Luque, and Ioannis Parissis.
\newblock Tauberian conditions, {M}uckenhoupt weights, and differentiation
  properties of weighted bases.
\newblock {\em Trans. Amer. Math. Soc.}, 367(11):7999--8032, 2015.

\bibitem{HagPar-SolyanikMaximal}
Paul Hagelstein and Ioannis Parissis.
\newblock Weighted solyanik estimates for the strong maximal function.
\newblock Preprint, arXiv:1410.3402 (2015).

\bibitem{Hruscev}
Sergei~V. Hru{\v{s}}{\v{c}}ev.
\newblock A description of weights satisfying the {$A_{\infty }$} condition of
  {M}uckenhoupt.
\newblock {\em Proc. Amer. Math. Soc.}, 90(2):253--257, 1984.

\bibitem{HP}
Tuomas Hyt\"onen and Carlos P\'erez.
\newblock Sharp weighted bounds involving ${A}_{\infty}$.
\newblock {\em Anal. PDE}, 6(4):777--818, 2013.

\bibitem{HPR1}
Tuomas Hyt{\"o}nen, Carlos P{\'e}rez, and Ezequiel Rela.
\newblock Sharp {R}everse {H}\"older property for {$A_\infty$} weights on
  spaces of homogeneous type.
\newblock {\em J. Funct. Anal.}, 263(12):3883--3899, 2012.

\bibitem{IM}
Tadeusz Iwaniec and Gaven Martin.
\newblock {\em Geometric function theory and non-linear analysis}.
\newblock Oxford Mathematical Monographs. The Clarendon Press, Oxford
  University Press, New York, 2001.

\bibitem{Kinnunen-Dissertationes}
Juha Kinnunen.
\newblock Sharp results on reverse {H}\"older inequalities.
\newblock {\em Ann. Acad. Sci. Fenn. Ser. A I Math. Dissertationes}, (95):34,
  1994.

\bibitem{Kinnunen-PubMat}
Juha Kinnunen.
\newblock A stability result on {M}uckenhoupt's weights.
\newblock {\em Publ. Mat.}, 42(1):153--163, 1998.

\bibitem{KLS}
A.~A. Korenovskyy, A.~K. Lerner, and A.~M. Stokolos.
\newblock On a multidimensional form of {F}. {R}iesz's ``rising sun'' lemma.
\newblock {\em Proc. Amer. Math. Soc.}, 133(5):1437--1440, 2005.

\bibitem{Kurtz_multiparameter}
Douglas~S. Kurtz.
\newblock Littlewood-{P}aley and multiplier theorems on weighted {$L^{p}$}\
  spaces.
\newblock {\em Trans. Amer. Math. Soc.}, 259(1):235--254, 1980.

\bibitem{LO}
Andrei~K. Lerner and Sheldy Ombrosi.
\newblock An extrapolation theorem with applications to weighted estimates for
  singular integrals.
\newblock {\em J. Funct. Anal.}, 262(10):4475--4487, 2012.

\bibitem{LOP1}
Andrei~K. Lerner, Sheldy Ombrosi, and Carlos P{\'e}rez.
\newblock Sharp {$A_1$} bounds for {C}alder\'on-{Z}ygmund operators and the
  relationship with a problem of {M}uckenhoupt and {W}heeden.
\newblock {\em Int. Math. Res. Not. IMRN}, (6):Art. ID rnm161, 11, 2008.

\bibitem{LOP2}
Andrei~K. Lerner, Sheldy Ombrosi, and Carlos P{\'e}rez.
\newblock Weak type estimates for singular integrals related to a dual problem
  of {M}uckenhoupt-{W}heeden.
\newblock {\em J. Fourier Anal. Appl.}, 15(3):394--403, 2009.

\bibitem{LP}
Teresa Luque and Ioannis Parissis.
\newblock Sharp weighted norm inequalities for multiparameter operators.
\newblock (preprint).

\bibitem{malak2001}
N.~A. Malaksiano.
\newblock On exact inclusions of {G}ehring classes in {M}uckenhoupt classes.
\newblock {\em Mat. Zametki}, 70(5):742--750, 2001.

\bibitem{malak2002}
Nikolay~Aleksandrovich Malaksiano.
\newblock The precise embeddings of one-dimensional {M}uckenhoupt classes in
  {G}ehring classes.
\newblock {\em Acta Sci. Math. (Szeged)}, 68(1-2):237--248, 2002.

\bibitem{MMNO}
J.~Mateu, P.~Mattila, A.~Nicolau, and J.~Orobitg.
\newblock B{MO} for nondoubling measures.
\newblock {\em Duke Math. J.}, 102(3):533--565, 2000.

\bibitem{Melas}
Antonios~D. Melas.
\newblock A sharp {$L^p$} inequality for dyadic {$A_1$} weights in
  {$\mathbb{R}^n$}.
\newblock {\em Bull. London Math. Soc.}, 37(6):919--926, 2005.

\bibitem{Muckenhoupt:Ap}
Benjamin Muckenhoupt.
\newblock Weighted norm inequalities for the {H}ardy maximal function.
\newblock {\em Trans. Amer. Math. Soc.}, 165:207--226, 1972.

\bibitem{OP-nondoubling}
Joan Orobitg and Carlos P{\'e}rez.
\newblock {$A_p$} weights for nondoubling measures in {${\mathbb R}^n$} and
  applications.
\newblock {\em Trans. Amer. Math. Soc.}, 354(5):2013--2033 (electronic), 2002.

\bibitem{PR-twoweight}
Carlos P{\'e}rez and Ezequiel Rela.
\newblock A new quantitative two weight theorem for the {H}ardy-{L}ittlewood
  maximal operator.
\newblock {\em Proc. Amer. Math. Soc.}, 143:641--655, 2015.

\bibitem{sjogren}
Peter Sj{\"o}gren.
\newblock A remark on the maximal function for measures in {${\bf R}^{n}$}.
\newblock {\em Amer. J. Math.}, 105(5):1231--1233, 1983.

\bibitem{sjogren-soria}
Peter Sj{\"o}gren and Fernando Soria.
\newblock Sharp estimates for the non-centered maximal operator associated to
  {G}aussian and other radial measures.
\newblock {\em Adv. Math.}, 181(2):251--275, 2004.

\end{thebibliography}

\end{document}